\documentclass[a4paper,reqno,12pt]{amsart}
\usepackage[latin1]{inputenc}
\usepackage[T1]{fontenc}
\usepackage[english]{babel}
\usepackage{appendix}
\usepackage[english]{minitoc}
\usepackage[nice]{nicefrac}
\usepackage[top=2.5cm, bottom=2.5cm, left=2.5cm,right=2.5cm]{geometry}
\numberwithin{equation}{section}
\makeatletter\@addtoreset{equation}{section} 
\usepackage{amsmath}
\usepackage{amssymb}
\usepackage{amsbsy} 
\usepackage{latexsym, amsfonts}
\usepackage{graphics}
\usepackage{ulem}
\usepackage{hhline}
\usepackage{dsfont}
\usepackage{mathrsfs}
\usepackage{color}
\usepackage{fancyhdr}
\usepackage{rotating}
\usepackage{fancybox}
\usepackage{colortbl}
\usepackage{pifont}
\usepackage{setspace}
\usepackage{enumerate}
\usepackage{multicol}
\usepackage{varioref}
\usepackage{lmodern}
\usepackage{textcomp}
\usepackage{euscript}
\usepackage[pdftex]{hyperref}
\newtheorem {theorem}{Theorem}[section]            \newtheorem {lemma}[theorem]{Lemma}
\newtheorem {definition}[theorem]{Definition}   \newtheorem {corollary}[theorem]{Corollary}     \newtheorem {remark}[theorem]{Remark}
\newtheorem {proposition}[theorem]{Proposition}       
         
\newcommand{\C}{\mathbb C}       \newcommand{\Z}{\mathbb Z}

\begin{document}


\title{$(p,q)$-complex It\^o--Hermite polynomials}
\author[A. Benahmadi, A. Ghanmi]{A. Benahmadi, A. Ghanmi}


\address{Analysis, P.D.E. and Spctral Geometry, Lab. M.I.A.-S.I., CeReMAR, Department of Mathematics, P.O. Box 1014, Faculty of Sciences, Mohammed V University in Rabat, Morocco}

\maketitle

\begin{abstract}
We introduce two classes of $(p,q)$-It\^o--Hermite polynomials, the post-quantum analogs of the $q$-It\^o--Hermite polynomials introduced recently by Ismail and Zhang. We study their basic properties such as their operational formulas of Rodrigues type, the corresponding raising and lowering operators as well as their generating functions and the $(p,p$-differential equations they obey. 
\end{abstract}



\section{Introduction}
The complex It\^o--Hermite polynomials, firstly considered by It\^o in his study of complex multiple Wiener integrals \cite{Ito52}, are extensively studied in the last decades  \cite{
	Shigekawa87,Matsumoto96,In,Gh13ITSF,Gh2018Mehler,
	AliBagarelloGazeau13,
IsmailTrans2016,RaichZhou04,DallingerRuotsalainenWichmanRupp10
}. They appear essentially in the concrete description of the spectral analysis of the Landau Hamiltonian and are applied in coherent
states, combinatorics, quantum optics,  nonlinear analysis of traveling-wave tube amplifiers, quasi-probabilities,
signal processing, singular values of the Cauchy transform, as well as in distribution of zeros of the automorphic reproducing kernel function.
Recently, M. Ismail and R. Zhang
have introduced and studied in
\cite{IsmailZhang2017} two
$q$-analogues of these polynomials which are functions of two complex variables, 
\begin{align}\label{qHmn1} H_{m,n}(z,w|q)= \sum\limits_{k=0}^{m\wedge n}\left[\!\!\!
\begin{array}{c}
m \\
k
\end{array}
\!\!\!\right]_{q} \left[\!\!\!
\begin{array}{c}
n \\
k
\end{array}
\!\!\!\right]_{q}(-1)^kq^{\binom{k}{2}}<1;q>_k z^{m-k}_1z^{n-k}_2
\end{align}
and
\begin{align}\label{qHmn2}
h_{m,n}(z,w|q)= \sum\limits_{k=0}^{m\wedge n}\left[\!\!\!
\begin{array}{c}
m \\
k
\end{array}
\!\!\!\right]_{q}\left[\!\!\!
\begin{array}{c}
n \\
k
\end{array}
\!\!\!\right]_{q}(-1)^k  q^{(m-k)(n-k)}<1;q>_k (z)^{m-k}(w)^{n-k}
\end{align}
for which they provide the raising and lowering operators together with Sturm-Liouville equations which
they obey. They also discussed the orthogonality relations, and Rodrigues formulas and multilinear generating functions for both classes.

In the present work, we combine the two classes in the following one, by introducing the parameter $p$, 
\begin{align}\label{pqHmn1} H_{m,n}(z,w|p,q)= \sum\limits_{k=0}^{m\wedge n}\left[\!\!\!
\begin{array}{c}
m \\
k
\end{array}
\!\!\!\right]_{p,q}\left[\!\!\!
\begin{array}{c}
n \\
k
\end{array}
\!\!\!\right]_{p,q}(-1)^kp^{(n-k)(m-k)}q^{\binom{k}{2}}<1;p,q>_k z^{m-k}_1z^{n-k}_2
\end{align}
so that for specific $p$ and $q$, we recover those considered in \cite{IsmailZhang2017} (defined through \eqref{qHmn1} and \eqref{qHmn2}, respectively). Indeed, we have $H_{m,n}(z,w|1,q)=H_{m,n}(z,w|q)$ and $H_{m,n}(z,w|p,1)=(-1)^m h_{m,n}(-z,w|p)= (-1)^n h_{m,n}(z,-w|p)$.
Thus, this class will be considered as the $(p,q)$-analogue of $H_{m,n}(z,w|q)$ and $h_{m,n}(z,w|p)$.

Another class  of $(p,q)$-It\^o Hermite polynomials, to be considered and studied,  is obtained by permuting the roles of $p$ and $q$ in the definition of the first considered class, to wit 
\begin{align}\label{pqHmn2}
G_{m,n}(z,\overline{z}|p,q)= \sum\limits_{k=0}^{m\wedge n}(-1)^k p^{\binom{k}{2}} q^{(m-k)(n-k)}<1;p,q>_k \left[\!\!\!
\begin{array}{c}
m \\
k
\end{array}
\!\!\!\right]_{p,q}\left[\!\!\!
\begin{array}{c}
n \\
k
\end{array}
\!\!\!\right]_{p,q} z^{m-k}w^{n-k}.
\end{align}

 The results we obtain here concern the $(p,q)$-It\^o--Hermite polynomials in \eqref{pqHmn1} and \eqref{pqHmn2}. Mainly, we study in Section 3 the basic properties of $H_{m,n}(z,\overline{z}|p,q)$, including their operational formulas of Rodrigues type, and derive the raising and lowering operators. Those of $G_{m,n}(z,\overline{z}|p,q)$ are discussed briefly in Section 4. We devote Section 5 to   derive  the  generating functions for $H_{m,n}(z,\overline{z}|p,q)$ and $_{m,n}(z,\overline{z}|p,q)$. 
To this end, we begin by describing and fixing in Section 2 needed notations, tools and basic  properties from $q$- and $(p,q)$-calculus, and introducing further tools, like the $(p,q)$-exponential like function. We also prove a  $(p,q)$-analogue of the classical $q$-binomial theorem
\begin{align*}
\sum\limits_{n=0}^{\infty}\frac{(a;q)_n}{(q;q)_n}z^n=\frac{(az;q)_\infty}{(z;q)_\infty}; \, |z|<1,
\end{align*}
due to Rothe and rediscovered by Cauchy \cite{adiga1985ramanujan}. 
This will play a crucial role in obtaining the analogues results of \cite{IsmailZhang2017} in  the $(p,q)$-setting. 

\section{$(p,q)$-Calculus }

The post-quantum calculus has many applications in a number of areas,  see e.g. \cite{CJORTPQA1991,MAKSARPQABSO2015,IAABPQBKO2018,AAMKMBO2018}, and cannot be obtained directly by replacing $q$ by $q/p$ in $q$-calculus as quoted in \cite{GVBDOBTP2016}. In this section, we fix all the classical  notations and terminology for  $(p,q)$-calculus  \cite{GVBDOBTP2016,graabpq2018}.					
 We also introduce some new ones. Those for $q$-calculus and corresponding to $p=1$ are the ones adopted in \cite{ETMQC2003}.

\subsection{Notations.}
The factorial function $[n]_{p,q}!$ of the $(p,q)$-analogue $[a]_{p,q}$, for a complex number $a$ with $p\neq q$, is defined by
\begin{align*}
[0]_{p,q}!=1 ;; \quad [n]_{p,q}!=\prod\limits_{k=0}^{n}[k]_{p,q};  \quad [a]_{p,q}=\frac{p^a-q^a}{p-q}.
\end{align*}
The $(p,q)$-shifted factorial is defined by
\begin{align*}
<a;p,q>_n &:= \prod\limits_{i=0}^{n-1}(p^{a+i}-q^{a+i})
\end{align*}
so that for $a=1$ we get 
\begin{align*}
<1;p,q>_n = (p-q)^{n}[n]_{p,q}! .
\end{align*}
While for $n=\infty$ and  $0<|q|<|p|\leq 1$ we set
\begin{align*}
<a;p,q>_\infty &:= \prod\limits_{i=0}^{\infty}(p^{a+i}-q^{a+i}).
\end{align*}
This defines the $(p,q)$-analogue of
$q$-shifted factorial be given by
\begin{align*}
 \quad <a;q>_n &:= \prod\limits_{i=0}^{n-1}(1-q^{a+i}); \quad
<a;q>_\infty :=\prod\limits_{i=0}^{\infty}(1-q^{a+i})
\end{align*}
with the convention that $<a;q>_0=1$ and $0<|q|<1$ when $n=\infty$.
The Watson notation will also be used. Thus, $ (a;q)_n=1$ and
\begin{align*}
(a;q)_n := \prod\limits_{i=0}^{n-1}(1-q^{i}a); \quad
(a;q)_\infty := \prod\limits_{i=0}^{\infty}(1-q^{i}a);  0<|q|<1.
\end{align*}
Therefore, the $(p,q)$-Gauss binomial coefficients are defined by
 \begin{align*}
\left[\!\!\!
  \begin{array}{c}
    m \\
    k
  \end{array}
  \!\!\!\right]_{p,q}:= \frac{<1;p,q>_m}{<1;p,q>_k<1;p,q>_{m-k}} .
\end{align*}
The $(p,q)$-power basis is defined by $(x\oplus y)^n_{p,q}=1$ for $n=0$ and 
 \begin{equation}
 (x\oplus y)^{(n)}_{p,q}=\prod\limits_{i=0}^{n-1}(p^ix+q^iy)= \sum\limits_{k=0}^n \left[\!\!\!
 \begin{array}{c}
 n \\
 k
 \end{array}
 \!\!\!\right]_{p,q} p^{\binom{k}{2}} q^{\binom{n-k}{2}} x^k y^{n-k} ; \quad n=1,2, \cdots.
 \end{equation}
 Another tool from the standard $(p,q)$-notation to be used is the $(p,q)$-derivative of a given function $f$ with respect to $x$ given as
	$$ \displaystyle D_{p,q,x}(f)(x)=
	\left \{ \begin{array}{ll}
\displaystyle	\frac{f(px)-f(qx)}{(p-q)x} , &  \mbox{if } \, x\neq 0,\\
	f'(0) , &  \mbox{if } \, x= 0.
	\end{array}
	\right.$$
Then, it is clear that for $\varphi_k(x):= x^k$; $k=1,2,\cdots$, we have 
	$$ \displaystyle D_{p,q,x}(\varphi_k)(x)= [k]_{p,q} \varphi_{k-1}(x) . 
	$$
Moreover, the $(p,q)$-Leibniz formula for the $(p,q)$-derivative operator reads
$$ D^m_{p,q}(f(x)g(x))=\sum\limits_{k=0}^m \left[\!\!\!
\begin{array}{c}
m \\
k
\end{array}
\!\!\!\right]_{p,q}  D^{m-k}_{p,q}(f)(p^kx) D^k_{p,q}(g)(q^{m-k}x) .$$

We conclude this section by introducing new notations and notions.
Thus, analogically to Watson notation, we define its $(p,q)$-extension by setting
\begin{align*}
(a;p,q)_n &:= \prod\limits_{i=0}^{n-1}(p^{i}-q^{i}a); \quad
(a;p,q)_\infty := \prod\limits_{i=0}^{\infty}(p^{i}-q^{i}a);  \ \  0<|q|<|p|\leq 1 .
\end{align*}
So that
$$ p^n \left(\frac{q}{p}; p,q\right)_n =  \left<\frac{q}{p};p,q\right>_n .$$
Therefore, with respect to this notation the $(p,q)$-Gauss binomial coefficient reads 
 \begin{align*}
\left[\!\!\!
  \begin{array}{c}
    m \\
    k
  \end{array}
  \!\!\!\right]_{p,q}&=\prod\limits_{k=0}^{n}\frac{\left(\frac{q}{p}; p,q\right)_n}{\left(\frac{q}{p}; p,q\right)_k \left(\frac{q}{p}; p,q\right)_{n-k}} .
\end{align*}

The next notion is defined as a possible extension of the Pochhammer symbol to the post-quantum calculus.

\begin{definition} We define  the $(p,q)$-Pochhammer symbol by
$$(a,b|p;q)_n=\prod\limits_{i=0}^{n-1}(p^i a-q^ib)$$
\end{definition}
Notice, for instance; that the relation to the previous $(p,q)$-shifted factorials are given by
$$  <a;p,q>_n=(p^a,q^a|p;q)_n \quad \mbox{and} \quad  (a;p,q)_n= (1,q^a|p;q)_n.$$
While the $(\frac{1}{p},\frac{1}{q})$ and $(p,q)$- Gauss binomial coefficients are connected by
 $$ \left[\!\!\!
  \begin{array}{c}
    m \\
    k
  \end{array}
  \!\!\!\right]_{\frac{1}{p},\frac{1}{q}}=\left[\!\!\!
  \begin{array}{c}
    m \\
    k
  \end{array}
  \!\!\!\right]_{p,q}(pq)^{k^2-km}.$$
 
 \subsection{A $(p,q)$-binomial theorem.} \label{s3}
 
 \begin{proposition}\label{binthm}
 	We have 
 	$$ \sum\limits_{n=0}^{\infty}\frac{(p,ap|p;q)_n}{(p,q|p;q)_n}z^n=\frac{(az,p;q)_\infty}{(z,p;q)_\infty}$$
 and	$$ \sum\limits_{n=0}^{\infty}\frac{(p,ap|p;q)_n}{(p,bq|p;q)_n}\left(\frac{z}{w} \right)^n=\frac{(bw,az|p;q)_\infty}{(w,z|p;q)_\infty}.$$
 \end{proposition}

 \begin{proof}
 	Consider the infinite product 
 	$$ F\left( \frac{z}{w}\right) =\prod\limits_{i=0}^{\infty} \frac{p^ibw-q^iaz}{p^iw-q^iz} $$
 	which is clearly uniformly convergent for fixed $a$, $b$, $p$ and $q$
 	inside $|z|<|w|$. It can be expanded as 
 	
 	$$ F\left( \frac{z}{w}\right)  = \sum\limits_{n=0}^{\infty}A_n\left(\frac{z}{w} \right)^n $$
 	for given  $A_n=A_n(b,a|p,q)$, and defines an analytic function in $\frac{z}{w}$  inside $|z|<|w|$. Moreover, 
 	\begin{align*}
 	(w-z)F\left( \frac{z}{w}\right) &=(bw-az)\prod\limits_{n=0}^{\infty} \frac{p^{n+1}bw-q^{n+1}az}{p^{n+1}w-q^{n+1}z} =(bw-az)F\left( \frac{qz}{pw}\right) .
 	\end{align*}
 	This implies that 
 	\begin{align} \label{idid}
 	\left( 1-\frac{z}{w}\right) F\left( \frac{z}{w}\right) &=\left( b-a\frac{z}{w}\right) F\left( \frac{qz}{pw}\right).  
 	\end{align}
 	Hence, $A_0=F(0)=1$ and by comparing the coefficient of  $\left(\frac{z}{w} \right)^n $ in the expansions of the involved quantities in \eqref{idid}, 
 	we see that
 	\begin{align*}
 	A_{n+1}-A_n=b\left(\frac{q}{p}\right)^{n+1}A_{n+1}-a\left(\frac{q}{p}\right)^{n}A_{n}
 	\end{align*}
 	Therefore, it follows 
 	\begin{align*}
 	A_n&=\prod\limits_{i=0}^{n-1} \frac{(1-(\frac{q}{p})^ia)}{(1-(\frac{q}{p})^{i+1}b)} =\frac{(p,pa|p,q)_n}{(p,bq|p,q)_n}.
 	\end{align*}
 \end{proof}

 \begin{remark}\label{remproof} For the particular case of $a=0$ and $b=w=1$, we obtain
 	$$\sum\limits_{n=0}^{\infty}p^{\binom{n+1}{2}} \frac{z^n}{(p,q|p,q)_n}=\frac{(1,0|p;q)_\infty}{(1,z|p;q)_\infty}=\frac{1}{\left(1,z|1,\frac{q}{p}\right)_{\infty}}.$$
 	While for the particular case $a=a/c$,$z=cz$ and $b=w=1$, we get 
 	$$\sum\limits_{n=0}^{\infty}q^{\binom{n}{2}} \frac{(pz)^n}{(p,q|p,q)_n}=\left(1,-z|1,\frac{q}{p}\right)_{\infty}$$
 \end{remark}
 
 \subsection{$(p,q)$ analog of the exponential function.}

By replacing the $q$-shifted factorial by their $(p,q)$ analogs in the $q$-exponential functions suggested by Euler \cite{ELIAI1748}, 
we can define their natural extension to the $(p,q)$ setting. Namely, 
\begin{align*}
e_{p,q}(z):&=\sum\limits_{n=0}^\infty  \frac{p^{\binom{n}{2}}}{<1;p,q>_n} z^n, \ \ |z| < p, \ \  0 < |q| < 1,
\end{align*}
and 
\begin{align*}
e_{\frac{1}{p},\frac{1}{q}}(z):&=\sum\limits_{n=0}^\infty  \frac{q^{\binom{n}{2}}}{<1;p,q>_n} z^n, \ \  0 < |q| < 1.
\end{align*}
More generally, we can propose a general version of the generalized $(p,q)$-exponential function.
\begin{definition} We define 
	\begin{align}\label{Expab}
	E_{p,q}^{x,a,b}(z):&=\sum\limits_{n=0}^\infty \frac{p^{a\binom{n}{2}} q^{b\binom{n}{2}} }{<x;p,q>_n} z^n,
	\end{align}
so that $E_{p,q}^{1,1,0}= e_{p,q}$ for $x=1=a$ and $b=0$ and $E_{p,q}^{1,0,1}= e_{\frac{1}{p},\frac{1}{q}}$ for $x=1=b$ and $a=0$.  
\end{definition}

The convergence of the above power series depends in the choice of the parameters $a,b,x,p$ and $q$.  
For example for $x=a=1$ and $b=0$, 
we distinguish two cases according to $|p|>|q|$ or not. Thus, if $|p|>|q|$ then $E_{p,q}^{1,0,1}= e_{1/p,1/q}$ is well defined on the whole complex plane, but when $|p|<|q|$ it is well defined only on the disc centered at the origin of radius $|q|$. Similarly, 
the series defining $E_{p,q}^{1,1,0}= e_{p,q}$
converges on $\C$ when $|p|<|q|$ and on the disc of radius $|p|$ when $|p|>|q|$. 
The complete description of the convergence problem of the series in \eqref{Expab} will be discussed elsewhere. The basic property we need to develop the rest of the paper is the explicit action of the $(\alpha,`beta)$-derivation of this function. Namely, we assert

\begin{lemma}\label{lemDerExp} We have
\begin{align}
D_{\alpha,\beta}\left[ E_{p,q}^{x,a,b}(\gamma z)\right]=\frac{\gamma}{p-q}\sum\limits_{n=0}^\infty  
\frac{\alpha^{n+1}-\beta^{n+1}}{p^{n+x}-q^{n+x}}
\frac{p^{a\binom{n}{2}} q^{b\binom{n}{2}} }{<x;p,q>_n} (p^aq^b\gamma z)^n.
\end{align} More generally, 
\begin{align}
D^m_{\alpha,\beta}\left[ E_{p,q}^{x,a,b}(\gamma z)\right]=\frac{\gamma^m (p^aq^b)^{\binom{m}{2}}}{(p-q)^m}\sum\limits_{n=0}^\infty  
\frac{<n+1;\alpha,\beta>_m}{<n+x;\alpha,\beta>_m}
\frac{(p^aq^b)^{\binom{n}{2}+mn}}{<x;p,q>_n} (\gamma z)^n.
\end{align}
\end{lemma}

\begin{proof}
	The proof follows by direct computation and next using mathematical induction combined with the fact that
	$$
	 \prod\limits_{k=0}^{m-1}A^{\alpha,\beta}_{p,q}(x,n+k)(p^aq^b)^{n+k} =(p^aq^b)^{\binom{m}{2}+mn}\frac{<n+1;\alpha,\beta>_m}{<n+x;\alpha,\beta>_m}. $$ 
\end{proof}

The following result is also needed
\begin{lemma}
	We have 
	\begin{align} \label{product}
	e_{p,q}(z)e_{\frac{1}{p},\frac{1}{q}}(w)=\sum_{n=0}^{\infty} \frac{(z\oplus w)^n_{p,q}}{<1;p,q>_n}. \end{align}
\end{lemma} 

\begin{proof}
Direct computation stating from $E_{p,q}^{x,a,b}(z) E_{p,q}^{x,c,d}(w)$ infers that
$$ E_{p,q}^{x,a,b}(z) E_{p,q}^{x,c,d}(w) = \sum_{m=0}^\infty \frac{1}{<x;p,q>_m} \left(
\sum_{n=0}^m 
\left[\!\!\!
\begin{array}{c}
	m \\ n
\end{array}
\!\!\!\right]_{p,q}^x
(p^aq^b)^{\binom{m}{2} + \binom{m-n}{2}} z^n w^{m-n} \right) ,$$
where 
$$ \left[\!\!\!
\begin{array}{c}
m \\ n
\end{array}
\!\!\!\right]_{p,q}^x := \frac{<x;p,q>_m}{<x;p,q>_n<x;p,q>_{jm-n}},$$
so that for $x=a=d=1$ and $b=c=0$, we get 
$$ E_{p,q}^{1,1,0}(z) E_{p,q}^{1,0,1}(w) = \sum_{n=0}^{\infty} \frac{(z\oplus w)^{(n)}_{p,q}}{<1;p,q>_n} .$$
This is exactly \eqref{product}.
\end{proof}

\begin{remark} 
For $w=-z$ we get $e_{p,q}(z)e_{\frac{1}{p},\frac{1}{q}}(-z)=1 $, since $(z\oplus w)^{(n)}_{p,q}= $ for every $n=1,2,\cdots$, when $w=-z$.  This means that the inverse of $e_{p,q}(z)$ is $e_{\frac{1}{p},\frac{1}{q}}(-z)$.
\end{remark}

  \section{(p,q)-It\^o--Hermite polynomials of two complex variables}
The first class of $(p,q)$ It\^o--Hermite polynomials of two complex variables is defined by 
   \begin{align}\label{pdIHpol1}
   H_{m,n}(z,w|p,q)= \sum\limits_{k=0}^{m\wedge n} (-1)^kp^{(n-k)(m-k)}q^{\binom{k}{2}}<1;p,q>_k  \left[\!\!\! \begin{array}{c} m \\ k \end{array}
   \!\!\!\right]_{p,q}\left[\!\!\!
   \begin{array}{c} n \\  k \end{array}
   \!\!\!\right]_{p,q} z^{m-k} w^{n-k} 
  \end{align}
 It can be considered as the $(p,q)$-analog of the $q$-polynomials studied by Ismail and Zhang in \cite{IsmailZhang2017} and which can be recovered as particular case by taking $p=1$.    
Direct computation using  \eqref{pdIHpol1} shows that these polynomials satisfy some lowering relations with respect to  $D_{\alpha,\beta}^\xi$ for specific $\alpha$ and $\beta$. The label $\xi$ is added to mean that the $(\alpha,\beta)$-derivation is taken with respect to the variable $\xi$ in the test function. More precisely, we have
\begin{align}
D_{p,q}^z(H_{m,n}(z,w|p,q))&=[m]_{p,q}H_{m-1,n}(z,pw|p,q)\\
D_{p,q}^w(H_{m,n}(z,w|p,q))&=[n]_{p,q}H_{m,n-1}(pz,w|p,q) \\
D_{\frac{1}{p},\frac{1}{q}}^z(H_{m,n}(z,w|p,q))&=[m]_{p,q}H_{m-1,n}\left( \frac{z}{pq} ,pw|p,q\right) \\
D_{\frac{1}{p},\frac{1}{q}}^w(H_{m,n}(z,w|p,q))&=[n]_{p,q}H_{m,n-1}\left( pz,\frac{w}{pq} |p,q\right) .
\end{align}
The polynomials 
\begin{align}\label{pdIHpol2}
H^{mod}_{m,n}(z,w|p,q) &= p^{-mn} \sum\limits_{k=0}^{m\wedge n}(-1)^kp^{k^2}q^{\binom{k}{2}}<1;p,q>_k \left[\!\!\!
\begin{array}{c} m \\ k \end{array}
\!\!\!\right]_{p,q}\left[\!\!\!
\begin{array}{c} n \\ k \end{array}
\!\!\!\right]_{p,q} z^{m-k}w^{n-k}  
\end{align}
are defined as adjacent class of $(p,q)$-Hermite polynomials of two variables.
They are connected to the first ones by  
\begin{align}\label{ModPol}
H_{m,n}^{mod}\left(z,w|p,q\right) :=  H_{m,n}
\left( \frac{z}{p^n} ,\frac{w}{p^{m}}|p,q
\right)
\end{align}
and
\begin{align}\label{ModPolff}   H^{mod}_{m,n}(z,w|p,q) = p^{(m-n)^2} H_{m,n}\left(\frac{z}{p^m} , \frac{w}{p^n}\bigg|p,q\right).
\end{align}
Both \eqref{ModPol} and \eqref{ModPolff} can be used to show  the following identity 
\begin{align}
H_{m,n}(z,w|p,q) 
&=p^{-(m-n)^2}H_{m,n}(p^{m-n}z,p^{n-m}w|p,q).
\end{align}

\begin{remark}
The infinite discrete set $\{p^{k(m-n)}z,p^{k(n-m)}w; \, k\in \Z\}$ is included in the set of zeros of $H_{m,n}(z,w|p,q)$, for fixed $p$ and $q$,  whenever $(z_0,w_0)$ is.
\end{remark}

The lowering relations for $H^{mod}_{m,n}(z,w|p,q)$ read simply as
 \begin{align}
 D_{p,q}^z(H^{mod}_{m,n}(z,w|p,q))&=[m]_{p,q} H^{mod}_{m-1,n}(z,w|p,q) \label{Dzmod}\\
 D_{p,q}^w(H^{mod}_{m,n}(z,w|p,q))&=[n]_{p,q} H^{mod}_{m,n-1}(z,w|p,q) \\
 D_{\frac{1}{p},\frac{1}{q}}^z(H^{mod}_{m,n}(z,w|p,q))&= [m]_{p,q} H^{mod}_{m-1,n}\left( \frac{z}{pq}  ,w|p,q\right) \\
 D_{\frac{1}{p},\frac{1}{q}}^w(H^{mod}_{m,n}(z,w|p,q))&= [n]_{p,q} H^{mod}_{m,n-1}\left( z,\frac{w}{pq}  |p,q\right) .
 \end{align}
These identities can also be derived from the operational representation of $H_{m,n}(z,w|p,q)$. Namely, these polynomials possess  two variants of the $(p,q)$-Rodrigues type formulas. To exact statement we perform the first order $(p,q)$-differential operator 
\begin{align}\label{CreatOp} \mathcal{D}^{\alpha, m_{w},n_{z}}_{\frac{1}{p},\frac{1}{q}}(f)(z,w) :=  \left[D^w_{\frac{1}{p},\frac{1}{q}}\right]^m \left[D^z_{\frac{1}{p},\frac{1}{q}}\right]^n\left(e_{\frac{1}{p},\frac{1}{q}}( \alpha zw) f\right) .
\end{align} 
 
\begin{theorem}\label{ThmRpdIH1}
The polynomials $H_{m,n}\left(z,w|p,q\right)$ possess the following operational realizations
\begin{align}\label{OF1}
e_{\frac{1}{p},\frac{1}{q}}(-qzw) H_{m,n}\left(z,w|p,q\right)
&= A_{m,n}^{p,q}  
\left[\mathcal{D}^{-qp^{m+n}, m_{w},n_{z}}_{\frac{1}{p},\frac{1}{q}}(1)\right](z,w)
\end{align}
and
\begin{align}\label{OF2}
e_{\frac{1}{p},\frac{1}{q}}\left( \frac{\alpha zw}{p^{m+n}}\right) H^{mod}_{m,n}\left(z,w|p,q\right)
&=B_{m,n}^{p,q}
\left[\mathcal{D}^{-q, m_{w},n_{z}}_{\frac{1}{p},\frac{1}{q}}(1)\right](z,w),
\end{align}
where the constants $A_{m,n}^{p,q}$ and $B_{m,n}^{p,q} $ are given by 
\begin{align*}
A_{m,n}^{p,q}:= \frac{(-1)^{m+n}  q^{mn} (p-q)^{m+n}}{  (pq)^{m+n} p^{\binom{m}{2}+\binom{n}{2}}}
\end{align*} 
and 
\begin{align*} 
 B_{m,n}^{p,q} := \frac{(pq)^{mn} (p-q)^{m+n} p^{\binom{m}{2}+\binom{n}{2}}}{q^{m+n}}.
\end{align*}
\end{theorem}

The proof lies essentially on the following lemma giving explicit expression of the $(1/p,1/q)$-derivative of the specific $(1/p,1/q)$-exponential function. This can be obtained as immediate consequence of Lemma \ref{lemDerExp}. However, we present here a direct proof.

 \begin{lemma}\label{lem1} For every complex number $\alpha$, we have 
	\begin{align}\label{lemnalpha} D^n_{\frac{1}{p},\frac{1}{q}}(e_{\frac{1}{p},\frac{1}{q}}(\alpha z))=(\frac{1}{p})^{\binom{n}{2}}\left(\frac{\alpha}{p-q}\right)^n e_{\frac{1}{p},\frac{1}{q}}\left( \frac{-qzw}{p^n}\right) .
	\end{align}	
\end{lemma}

\begin{proof}
	Using the fact that $D_{\alpha,\beta} (z^k) = [k]_{\alpha,\beta} z^{k-1}$; for $k\geq 0$ with the convention that $D_{\alpha,\beta} (z^0)=0$, as well as the facts $$[n+1]_{\frac{1}{p},\frac{1}{q}}
	={p^{n+1}-q^{n+1}}/{(pq)^n(p-q)}$$ and $$<1;p,q>_{n+1} = (p^{n+1}-q^{n+1})<1;p,q>_{n},$$
	 we get
	\begin{align*}
	D_{\frac{1}{p},\frac{1}{q}}^z(e_{\frac{1}{p},\frac{1}{q}}(\alpha zw))
	&=\sum\limits_{n=0}^\infty  \frac{q^{\binom{n}{2}} }{<1;p,q>_n} (\alpha w)^n D_{\frac{1}{p},\frac{1}{q}}^z (z^n)
\\&	=\sum\limits_{n=0}^\infty [n+1]_{1/p,1/q} \frac{ q^{\binom{n+1}{2}} }{<1;p,q>_{n+1}} (\alpha w)^{n+1}z^{n}.
\\	&=\frac{\alpha w}{p-q}\sum\limits_{n=0}^\infty  \frac{q^{\binom{n}{2}}}{<1;p,q>_{n}}\left(\frac{\alpha z w}{p}\right)^{n}.
	\end{align*}
Therefore, the result in \eqref{lemnalpha} follows by induction.
\end{proof}

\begin{proof}[Proof of Theorem \ref{ThmRpdIH1}]
 We have to prove that 
	$$ \mathcal{D}^{m_{w},n_{z}}_{\frac{1}{p},\frac{1}{q}}\left( e_{\frac{1}{p},\frac{1}{q}}\left( \alpha zw\right) \right) = \sum\limits_{k=0}^{m\wedge n}
	\frac{p^{k^2}q^{\binom{k+1}{2}}}{\alpha^{k}}<1;p,q>_k
	\left[\!\!\!
	\begin{array}{c} m \\ k \end{array}
	\!\!\!\right]_{p,q}\left[\!\!\!
	\begin{array}{c} n \\  k  \end{array}
	\!\!\!\right]_{p,q}  z^{m-k}_1z^{n-k}_2$$
for arbitrary $\alpha$,	so that for the specific $\alpha=-q$, we get 
	\begin{align*}
\mathcal{D}^{m_{w},n_{z}}_{\frac{1}{p},\frac{1}{q}}\left( e_{\frac{1}{p},\frac{1}{q}}(-qzw)\right) &= 	
p^{mn}	H_{m,n}\left(\frac{z}{p^n} ,\frac{w}{p^m}|p,q\right)= p^{mn} H^{mod}_{m,n}(z ,w|p,q) . 
\end{align*}
	The proof follows by applying Lemma \ref{lem1}. Indeed, the appropriate application of the $(1/p,1/q)$-Leibniz formula for the $(1/p,1/q)$-derivative operator and its explicit expression when acting on the  $(1/p,1/q)$-exponential function, we obtain
		$$ \left[D^w_{\frac{1}{p},\frac{1}{q}}\right]^m \left[D^z_{\frac{1}{p},\frac{1}{q}}\right]^n\left( e_{\frac{1}{p},\frac{1}{q}}(\alpha zw)\right)
		= \frac{	e_{\frac{1}{p},\frac{1}{q}}\left( \frac{\alpha zw}{p^{m+n}}\right) }{p^{mn}\gamma_{m,n}^{p,q}}
	 \sum\limits_{k=0}^{m\wedge n}
	 \frac{p^{k^2}q^{\binom{k+1}{2}}}{\alpha^{k}} <1;p,q>_k  \left[\!\!\!
		\begin{array}{c}
		m \\
		k
		\end{array}
		\!\!\!\right]_{p,q}\left[\!\!\!
		\begin{array}{c}
		n \\
		k
		\end{array}
		\!\!\!\right]_{p,q} z^{m-k}_1z^{n-k}_2.
		$$
	\end{proof}

 \begin{remark}
	According to proof above, we can define a generalized class of these $(p,q)$-Ito-Hermite polynomials depending on the parameter $\alpha$.  
\end{remark}

In the sequel, let consider the elementary $(1/p,1/q)$-differential operator 
\begin{align}\label{CreatOpc} \square^{m,n,\xi}_{\frac{1}{p},\frac{1}{q}}(f)(\xi) =   \left[e_{\frac{1}{p},\frac{1}{q}}\left( \frac{\alpha\xi}{p^{m+n}}\right)  \right]^{-1}   D^\xi_{\frac{1}{p},\frac{1}{q}} \left( e_{\frac{1}{p},\frac{1}{q}}\left( \frac{\alpha\xi}{p^{m+n}}\right)\right) 
\end{align}
which, for $\alpha=-q$,  can be seen as the creation operator for the modified $(p,q)$-It\^o--Hermite polynomials $H_{m,n}^{mod}\left(z,w|p,q\right) $.
Accordingly, the operators $D^\xi_{\frac{1}{p},\frac{1}{q}}$ becomes the creation operators for the modified $(p,q)$-It\^o--Hermite functions defined by 
\begin{align}\label{ModFunc}
h_{m,n}^{mod}\left(z,w|p,q\right) :=  e_{\frac{1}{p},\frac{1}{q}}\left(\frac{-qzw}{p^{m+n}}\right) H_{m,n}\left(\frac{z}{p^n} ,\frac{w}{p^m}|p,q\right) .
\end{align}
 More precisely, we assert the following.

\begin{proposition} \label{propCre} We have
		\begin{align} \label{hopann}
	h_{m+1,n}^{mod}\left(z,w|p,q\right)  &=  \frac{(q-p) p^m q^n}{q}  
	D^w_{\frac{1}{p},\frac{1}{q}}
	\left( h_{m,n}^{mod}\left(z,w|p,q\right) \right) 	\end{align}
	and
	\begin{align}
	h_{m,n+1}^{mod}\left(z,w|p,q\right)  &= \frac{(q-p) p^m q^n}{q}    D^z_{\frac{1}{p},\frac{1}{q}}\left( h_{m,n}^{mod}\left(z,w|p,q\right) \right).
	\end{align}

\end{proposition}
\begin{proof}
The proof follows by making appeal to Theorem \ref{ThmRpdIH1} and applying the Rodrigues type formula.
\end{proof}

\begin{remark} The previous result reads in terms of $\square^{m,n,\xi}_{\frac{1}{p},\frac{1}{q}}$ in \eqref{CreatOpc} as 
$$
	H_{m+1,n}^{mod}\left(z,w|p,q\right)  =  \frac{(q-p) p^m q^n}{q}  \square^{m,n,w}_{\frac{1}{p},\frac{1}{q}}\left( H_{m,n}^{mod}\left(z,w|p,q\right) \right) $$ 
	and
$$
	H_{m,n+1}^{mod}\left(z,w|p,q\right)  = \frac{(q-p) p^m q^n}{q}  \square^{m,n,z}_{\frac{1}{p},\frac{1}{q}}\left( H_{m,n}^{mod}\left(z,w|p,q\right) \right).$$
\end{remark}

As immediate consequence, we derive three term recursion formulas.  
\begin{proposition} We have
	\begin{align*}
	H_{m+1,n}(z,w|p,q)&=p^nzH_{m,n}(z,w|p,q)-q^{m}(p^n-q^n)H_{m,n-1}(z,w|p,q)\\
	H_{m,n+1}(z,w|p,q)&=p^mwH_{m,n}(z,w|p,q)-q^{n}(p^m-q^m)H_{m-1,n}(z,w|p,q).
	\end{align*}
\end{proposition}

The proof is based on Proposition \ref{ThmRpdIH1} and the fact that 
$$ D^n_{\frac{1}{p},\frac{1}{q},z}e_{\frac{1}{p},\frac{1}{q}}\left(\alpha z\right)|_{\frac{z}{\beta}}=\beta^nD^n_{\frac{1}{p},\frac{1}{q},z}e_{\frac{1}{p},\frac{1}{q}}\left(\frac{\alpha z}{\beta}\right).$$
Another immediate consequence is that the  $H_{m,n}^{mod}\left(z,w|p,q\right)$ in \eqref{ModPol} appears as eigenfunctions of the second order post-quantum differential operator 
$$
[\Delta^{p,q}_{m,n} (f)](z,w) :=  
\left\{ 
D^z_{\frac{1}{p},\frac{1}{q}}
  D^w_{p,q}  -\frac{qw}{p^{m+n}}D^z_{p,q}\circ \aleph_{1/q}  \right\} (f)(z,w),$$
where we have set 
$$ \aleph _{1/q}f (z,w) :=  f\left( \frac{z}{q},w\right).$$

\begin{proposition} 
The modified polynomials $H_{m,n}^{mod}\left(z,w|p,q\right)$ obey the post-quantum differential equations
\begin{align} \label{Eigenmn} 
\Delta^{p,q}_{m,n} \left(  H^{mod}_{m,n}(z,w|p,q)\right)  = 
\frac{q [m]_{p,q}}{(q-p) p^{m-1} q^n} H^{mod}_{m,n}(z,w|p,q) .
\end{align}
\end{proposition}

\begin{proof}
	Notice first that the operator $\Delta^{p,q}_{m,n} $ can be factorized as 
	$$
	\Delta^{p,q}_{m,n} (f) :=  \left( e_{\frac{1}{p},\frac{1}{q}}\left( \frac{-qzw}{p^{m+n}}\right)\right)^{-1}
	D^w_{\frac{1}{p},\frac{1}{q}}
	\left( e_{\frac{1}{p},\frac{1}{q}}\left( \frac{-qzw}{p^{m+n-1}} \right) D^z_{p,q} f \right).
	$$
Thus, from \eqref{Dzmod}, \eqref{hopann} and the last realization of $\Delta^{p,q}_{m,n}$, we obtain \eqref{Eigenmn}. 
 \end{proof}
 
\section{The polynomials $G_{m,n}(z,\overline{z}|p,q)$} 
By permuting the roles of $p$ and $q$ in the definition of the first considered class, we define
$$G_{m,n}(z,\overline{z}|p,q)= \sum\limits_{k=0}^{m\wedge n}\left[\!\!\!
\begin{array}{c}
m \\
k
\end{array}
\!\!\!\right]_{p,q}\left[\!\!\!
\begin{array}{c}
n \\
k
\end{array}
\!\!\!\right]_{p,q}(-1)^k p^{\binom{k}{2}} q^{(m-k)(n-k)}<1;p,q>_k z^{m-k}w^{n-k}.$$
Subsequently, we can obtain easily the following 

\begin{proposition} We have
	\begin{align*}
	G_{m,n}(z,w|p,q)&= \frac{\left(q-p\right)^{m+n}}{p^{\binom{m}{2}+\binom{n}{2}}}
	\left( e_{p,q}(-p^{m+n}zw)\right)^{-1} D^m_{p,q,w}D^n_{p,q,z}(e_{p,q}(-zw))
	\end{align*}
	and
	\begin{align*}
	G_{m,n}(p^nz,p^mw|p,q)&= \frac{\left(q-p\right)^{m+n}}{p^{\binom{m}{2}+\binom{n}{2}-mn}}
	 \left( e_{p,q}(-p^{m+n}zw)\right)^{-1} D^m_{p,q,w}D^n_{p,q,z}(e_{p,q}(-zw))
	\end{align*} 
\end{proposition}

\begin{remark}
		The last identity reads equivalently 
	\begin{align*}
	G_{m,n}(z,w|p,q)&= \frac{\left(q-p\right)^{m+n}}{p^{\binom{m}{2}+\binom{n}{2}-mn}}\left( e_{p,q}(-zw) \right)^{-1} D^m_{p,q,w}D^n_{p,q,z}(e_{p,q}(-\frac{zw}{p^{m+n}} ))
	\end{align*}
\end{remark}

The direct proofs follow thanks to the facts
	\begin{align*}
	D^n_{p,q,z}(e_{p,q}(\alpha zw))&=p^{\binom{n}{2}}\left(\frac{\alpha w }{p-q}\right)^{n}e_{p,q}(\alpha p^{n}zw)
	\end{align*}
	and 
	\begin{align*}
	D^m_{p,q,w}D^n_{p,q,z}(e_{p,q}(\alpha zw))&=p^{\binom{m}{2}+\binom{n}{2}}\left(\frac{\alpha}{q-p}\right)^{m+n}e_{p,q}(\alpha p^{m+n}zw)\\
	& \times \sum\limits_{k=0}^{m\wedge n}
	\alpha^{-k}\left(\frac{1}{p}\right)^{\binom{k}{2}} \left(pq\right)^{(m-k)(n-k)}<1;p,q>_k \left[\!\!\!
	\begin{array}{c}
	m \\
	k
	\end{array}
	\!\!\!\right]_{p,q}\left[\!\!\!
	\begin{array}{c}
	n \\
	k
	\end{array}
	\!\!\!\right]_{p,q} z^{m-k}w^{n-k}.
	\end{align*}

The next result is the analogue of Proposition \ref{propCre}.

\begin{proposition} We have  lowering relations
	\begin{align*}
	D_{\frac{1}{p},\frac{1}{q},z}G_{m,n}(z,\overline{z}|p,q)&=\frac{[m]_{p,q}}{(pq)^{m-n-1}}G_{m-1,n}(z,\overline{z}|p,q) \\	
	D_{\frac{1}{p},\frac{1}{q},w}G_{m,n}(z,\overline{z}|p,q)&=\frac{[n]_{p,q}}{(pq)^{n-m-1}}G_{m,n-1}(z,\overline{z}|p,q)\\
	D_{p,q}^{z}h_{m,n}(z,w|p,q)&=[m]_{p,q}G_{m-1,n}(z,qw|p,q)\\
	D_{p,q}^{z}h_{m,n}(z,w|p,q)&=[n]_{p,q}G_{m,n-1}(qz,w|p,q).
	\end{align*}
\end{proposition}

We conclude this section by noticing that the two introduced classes are also closely connected by the identity 
	$$G_{m,n}(z,w|p,q)=(pq)^{mn}i^{m+n}H_{m,n}\left( iz,iw \bigg|\frac{1}{p},\frac{1}{q}\right). $$

\section{generating functions}

Set 
\begin{align}\label{GenFctH1}
R_{p,q}(z,w|u,v)&:= \sum\limits_{m,n=0}^{\infty} p^{\frac{(m-n)^2}{2}} \frac{u^m v^n}{<1;p,q>_m <1;p,q>_n}H_{m,n}(z,w|p,q)
\end{align}	
and 
\begin{align}\label{GenFcth}
S_{p,q}(z,w|u,v)&:=\sum\limits_{m,n=0}^{\infty}  q^{\frac{(m-n)^2}{2}} \frac{u^m v^n}{<1;p,q>_m <1;p,q>_n} G_{m,n}(z,w|p,q).
\end{align}

\begin{theorem} The closed expressions of the generating functions in \eqref{GenFctH1} and \eqref{GenFcth}  are given by 
\begin{align} \label{GenFctH1closed}
R_{p,q}(z,w|u,v) 
=e_{p,q}(\sqrt{p}vw)e_{p,q}(\sqrt{p}uz)e_{\frac{1}{p},\frac{1}{p}}(-uv)
	\end{align}
	and
	\begin{align} \label{GenFcthclosed}
	S_{p,q}(z,w|u,v)  
	=e_{\frac{1}{p},\frac{1}{p}}(\sqrt{q}vw)e_{\frac{1}{p},\frac{1}{p}}(\sqrt{q}uz)e_{p,q}(-uv).
	\end{align}
\end{theorem}

\begin{proof}  
For the proof of \eqref{GenFctH1closed},  
we apply the Rodrigues type formula for the $H_{m,n}(z,w|p,q)$ and the Lemma \ref{lem1},  we  get
	\begin{align*}
	R_{p,q}(z,w|u,v)&=\frac{1}{e_{\frac{1}{p},\frac{1}{p}}(-qzw)}\sum\limits_{m=0}^{\infty}\frac{(p-q)^m}{(pq)^m}\frac{u^m}{<1;p,q>_m}p^{\frac{m}{2}}D^m_w e_{\frac{1}{p},\frac{1}{p}}(-qp^{m}zw)\\
	&\times\left( \sum\limits_{n=0}^{\infty}\frac{(p-q)^np^{\frac{n}{2}}}{(pq)^n}\frac{v^n}{<1;p,q>_n}q^{mn}p^{-mn}\frac{(-qp^{m+n})^{n}}{(p-q)^n}p^{-\binom{n}{2}}w^n\right) \\
	&=\frac{1}{e_{\frac{1}{p},\frac{1}{p}}(-qzw)}\sum\limits_{m}^{\infty}\frac{(p-q)^m}{(pq)^m}\frac{u^m}{<1;p,q>_m}p^{\frac{m}{2}}D^m_w e_{\frac{1}{p},\frac{1}{p}}(-qp^{m}zw)e_{p,q}(pq^mvw).
	\end{align*}
	Now, the use of the  $(p,q)$-Leibniz formula for the $(p,q)$-derivative operator implies
	\begin{align*}
	R_{p,q}(z,w|u,v)&=\frac{1}{e_{\frac{1}{p},\frac{1}{p}}(-qzw)}\sum\limits_{m}^{\infty}\frac{(p-q)^m}{(pq)^m}\frac{u^m}{<1;p,q>_m}p^{\frac{m}{2}}\\
	&\sum\limits_{k=0}^{m}\left[\!\!\!
	\begin{array}{c}
	m \\
	k
	\end{array}
	\!\!\!\right]_{\frac{1}{p},\frac{1}{p}}\frac{(-qp^{m})^{m-k}}{(p-q)^{m-k}}p^{-\binom{m-k}{2}}z^{m-k}e_{\frac{1}{p},\frac{1}{p}}(-qzw)\frac{(\sqrt{p}q^mv)^k}{(p-q)^k}p^{-\binom{k}{2}}v^ke_{p,q}(\sqrt{p}vw)\\
	&=e_{p,q}(\sqrt{p}vw)\sum\limits_{m}^{\infty}\frac{u^m}{<1;p,q>_m}\\
	&\sum\limits_{k=0}^{m}\left[\!\!\!
	\begin{array}{c}
	m \\
	k
	\end{array}
	\!\!\!\right]_{\frac{1}{p},\frac{1}{p}}p^{m^2-mk+\frac{k}{2}-\frac{m}{2}}p^{-\binom{m-k}{2}}z^{m-k}q^{mk-k}q^{-\binom{k}{2}}(-v)^k\\
	&=e_{p,q}(\sqrt{p}vw)\sum\limits_{m}^{\infty}\frac{u^m}{<1;p,q>_m}
	\sum\limits_{k=0}^{m}\left[\!\!\!
	\begin{array}{c}
	m \\
	k
	\end{array}
	\!\!\!\right]_{p,q}p^{\binom{m-k}{2}}(\sqrt{p}z)^{m-k}q^{\binom{k}{2}}(-v)^k\\
	&=e_{p,q}(\sqrt{p}vw)e_{p,q}(\sqrt{p}uz)e_{\frac{1}{p},\frac{1}{p}}(-uv)
	\end{align*}
	
	Now, using  the $(p,q)$-Gauss binomial coefficient, we can rewrite $S_{p,q}(z,w|u,v)$ in \eqref{GenFcth} as 
	\begin{align*}
	S_{p,q}(z,w|u,v)
	&=\sum\limits_{m,n}^{\infty}\sum\limits_{k=0}^{m\wedge n}p^{\binom{k}{2}}\frac{(-uv)^k}{<1;p,q>_k}q^{(m-k)(n-k)}\frac{\left(uz\right)^{m-k}}{<1;p,q>_{m-k}}\frac{\left(vw\right)^{n-k}}{<1;p,q>_{n-k}}\\
	&=\sum\limits_{m,n}^{\infty}\sum\limits_{k=0}^{m\wedge n}p^{\binom{k}{2}}\frac{(-uvp)^k}{<1;p,q>_k}\frac{q^{\binom{m-k}{2}}\left(q^{\frac{1}{2}}uz\right)^{m-k}}{<1;p,q>_{m-k}}\frac{q^{\binom{n-k}{2}}\left(q^{\frac{1}{2}}vz\right)^{n-k}}{<1;p,q>_{n-k}}.
	\end{align*}
	Therefore, the generating function \eqref{GenFcthclosed} follows making use of the two formulas in Remark \ref{remproof}.
\end{proof}

The previous result can be rewritten in the following equivalent form thanks to to the 
$(p,q)$-binomial theorem. (Proposition \ref{binthm}). 

\begin{corollary} We have
	\begin{align}
	R_{p,q}(z,w|u,v)
	&=\frac{\left( 1,\frac{-uv}{p}\right|1,\frac{q}{p})_\infty}{\left( 1,\frac{uz}{p^{\frac{1}{2}}}|1,\frac{q}{p}\right)_\infty \left( 1,\frac{vw}{p^{\frac{1}{2}}}|1,\frac{q}{p}\right)_\infty}
	\end{align}
	and
		\begin{align} \label{GenFcthclosedc}
	S_{p,q}(z,w|u,v) 
	=\frac{\left(1,-\frac{q^{\frac{1}{2}}uz}{p} |1,\frac{q}{p}\right)_{\infty}\left(1,-\frac{q^{\frac{1}{2}}vz}{p} |1,\frac{q}{p}\right)_{\infty}}{\left(1,-uv|1,\frac{q}{p}\right)_{\infty}}.
	\end{align}
	
\end{corollary}


\begin{thebibliography}{99}
		\bibitem{AliBagarelloGazeau13}  Ali S.T., Bagarello F., Gazeau  J-P.,
{ Quantizations from reproducing kernel spaces}.
Ann. Physics 332 (2013), 127--142.


	\bibitem{DallingerRuotsalainenWichmanRupp10}  Dallinger R., Ruotsalainen H., Wichman R., Rupp M.,
{ Adaptive pre-distortion techniques based on orthogonal polynomials}.
In Conference Record of the 44th Asilomar Conference on Signals, Systems and Computers, IEEE (2010) pp 1945-1950.




\bibitem{Gh13ITSF}  Ghanmi A.,
Operational formulae for the complex Hermite polynomials $H_{p,q}(z, \bar z)$.
{ Integral Transforms Spec. Funct.}  2013; 24 (11):884-895.  

\bibitem{Gh2018Mehler}  Ghanmi A.,
{ Mehler's formulas for the univariate complex Hermite polynomials and applications}.
 Math. Methods Appl. Sci.


	\bibitem{In} Intissar.A, Intissar.A,{Spectral properties of the Cauchy transform on $L^{2}(\C,e^{-|z|^{2}}d\lambda(z)$.} J. Math. anal. Appl. 313, no 2 (2006) 400-418.

\bibitem{IsmailTrans2016} Ismail M.E.H.,
{ Analytic properties of complex Hermite polynomials}.
Trans. Amer. Math. Soc. 368 (2016), no. 2, 1189-1210.

\bibitem{Ito52}  It\^o K.,
{ Complex multiple Wiener integral}.
{ Jap. J. Math.}, 22 (1952) 63-86.






\bibitem{Matsumoto96}
Matsumoto H., Ueki N.,
Spectral analysis of Schr\"odinger operators with magnetic fields.
{ J. Funct. Anal.} (1) \textbf{140} (1996)  218--255.



	\bibitem{RaichZhou04} Raich R., Zhou G.,
{ Orthogonal polynomials for complex Gaussian processes}.
IEEE Trans. Signal Process., vol. 52 (2004) no. 10, pp. 2788-2797.



\bibitem{Shigekawa87}  
 Shigekawa I.,
{Eigenvalue problems of Schr\"odinger operator with magnetic field on compact Riemannian manifold},
{ J. Funct. Anal.} 75 (1987) 92-127.

	
	\bibitem{IsmailZhang2017} {Ismail, M.,  Zhang, R. (2017).} 
	{  On some $2$D orthogonal .}.
	 Transactions of the American Mathematical Society, 369(10), 6779-6821.
	
	 \bibitem{adiga1985ramanujan}{Adiga, Chandrashekar and Berndt, Bruce C and Bhargava, S and Watson, George Neville} 
	{Chapter 16 of  Ramanujan's Second Notebook: Theta-Functions and $ q $-Series: Theta Functions and Q-series}.
	 American Mathematical Soc, 53 (1985), v+85 pp.
	
	\bibitem{CJORTPQA1991}
	{Chakrabarti, R., \& Jagannathan, R. (1991).}
	{A (p, q)-oscillator realization of two-parameter quantum algebras.}
	 Journal of Physics A: Mathematical and General, 24(13), L711.
	\bibitem{MAKSARPQABSO2015}
	{Mursaleen, M., Ansari, K. J.,  Khan, A. (2015).}
	{Some approximation results by $(p, q)$-analogue of Bernstein Stancu operators.}
	 Applied Mathematics and Computation, 264, 392-402.
	\bibitem{IAABPQBKO2018}
	{Ilarslan, H. G. I.,  Acar, T. (2018).}
	{ Approximation by bivariate $(p, q)$-Baskakovâ Kantorovich operators.}
	 Georgian Mathematical Journal, 25(3), 397-407.
	\bibitem{AAMKMBO2018}
	{Acar, T., Aral, A.,  Mohiuddine, S. A. (2018).}
	 {On Kantorovich modification of (p, q)-Bernstein operators.}
	  Iranian Journal of Science and Technology, Transactions A: Science, 42(3), 1459-1464.
	
	\bibitem{MGMBFAA2018}
{Milovanoviki, G. V., Gupta, V., Malik, N. (2018). }
{$(p, q)$-Beta functions and applications in approximation.}
 Boleti­n de la Sociedad Matemitica Mexicana, 24(1), 219-237.


 \bibitem{ETMQC2003}
 {Ernst, T. (2003).}
 {A method for q-calculus.}
  Journal of Nonlinear Mathematical Physics, 10(4), 487-525.

 \bibitem{GVBDOBTP2016}
{Gupta, V. (2016).}
{Bernstein Durrmeyer operators based on two parameters.}
 Facta Universitatis, Series: Mathematics and Informatics, 31(1), 79-95.


 \bibitem{ELIAI1748}
 {Euler L,} {Introductio in Analysin Infinitorum,} T. 1, Chapter XVI, p. 259, Lausanne, 1748.
 	
 	
 	
 	\bibitem{graabpq2018}
 	{Gupta V., Rassias T.M., Agrawal P.N., Acu A.M. (2018)}
 	{Basics of Post-quantum Calculus.}
  Recent Advances in Constructive Approximation Theory. Springer Optimization and Its Applications, vol 138. Springer, Cham
	
	\end{thebibliography}
\end{document}